\documentclass{amsart}
\usepackage{amsfonts,amssymb,amsmath,amsthm}
\usepackage{url}
\usepackage{verbatim}
\usepackage{enumerate}
\usepackage{color}
\usepackage[color,all]{xy}
\usepackage{hyperref}
\usepackage[all]{xy}
\newtheorem{thrm}{Theorem}[section]
\newtheorem{lem}[thrm]{Lemma}

\newtheorem{cor}[thrm]{Corollary}

\theoremstyle{definition}
\newtheorem{example}[thrm]{Example}
\newtheorem{definition}[thrm]{Definition}
\newtheorem{remark}[thrm]{Remark}

\newcommand{\Sing}{\operatorname{Sing}}

\newcommand{\Pic}{\operatorname{Pic}}

\newcommand{\Hom}{\operatorname{Hom}}
\newcommand{\im}{\operatorname{Im}}
\newcommand{\Iden}{\operatorname{Id}}

\newcommand{\End}{\operatorname{End}}

\newcommand{\cF}{\mathcal{F}}
\newcommand{\Oc}{\mathcal{O}_C}
\newcommand{\cU}{\mathcal{U}}
\newcommand{\Und}{\cU_C (n, d)}
\newcommand{\PP}{\mathbb{P}}
\newcommand{\bV}{\mathbb{V}}
\newcommand{\bM}{\mathbb{M}}
\newcommand{\hX}{\widehat{X}}

\author{Parviz Asefi Nazarlou}
\address{P.\ A.\ N.: Department of Mathematics, Azarbaijan Shahid Madani University, Tabriz, I.\ R.\ Iran, P.\ O.\ box 53751-71379}
\email{p.asefinazarloo@azaruniv.ac.ir}

\author{Ali Bajravani}
\address{A.\ B.: Institut für Mathematik, Humboldt Universität zu Berlin, Germany\\
Department of Mathematics, Azarbaijan Shahid Madani University, Tabriz, I.\ R.\ Iran, P.\ O.\ box 53751-71379}
\email{ali.bajravani@hu-berlin.de; bajravani@azaruniv.ac.ir}

\author{George H.\ Hitching}
\address{G.\ H.\ H.: Oslo Metropolitan University, Postboks 4, St. Olavs plass, 0130 Oslo, Norway}
\email{gehahi@oslomet.no}



\numberwithin{equation}{section}

\begin{document}

\title{Martens and Mumford theorems for higher rank Brill--Noether loci}

\begin{abstract}
Generalizing the Martens theorem for line bundles over a curve $C$, we obtain upper bounds on the dimension of the Brill--Noether locus $B^k_{n, d}$ parametrizing stable bundles of rank $n \ge 2$ and degree $d$ over $C$ with at least $k$ independent sections. This proves a conjecture of the second author and generalizes bounds obtained by him in the rank two case. We give more refined results for some values of $d$, including a generalized Mumford theorem for $n \ge 2$ when $d \le g - 1$. The statements are obtained chiefly by analysis of the tangent spaces of $B^k_{n, d}$. As an application, we show that for $n\geq 5$ the locus $B^2_{n, n(g-1)}$ is irreducible and reduced for any $C$.
\end{abstract}

\maketitle

\section{Introduction}

Let $C$ be a complex projective smooth curve of genus $g$. We denote by $\Und$ the moduli space of stable bundles of rank $n$ and degree $d$ over $C$; it is an irreducible quasiprojective variety of dimension $n^2 (g-1) + 1$. For $1 \leq k \leq n + \frac{d}{2}$, the Brill--Noether subvariety $B^k_{n, d}$ is set-theoretically the locus of bundles in $\Und$ with at least $k$ independent sections. The locus $B^k_{n, d}$ is determinantal of expected dimension $n^2(g-1) + 1 -k (k - d + n(g-1))$, and naturally generalizes the familiar Brill--Noether loci $W_d^{k-1} (C) = B^k_{1, d}$ on the Picard variety $\Pic^d (C)$.

A fundamental result in Brill--Noether theory is Martens' theorem \cite[p.\ 191]{ACGH} (see also Lemma \ref{ExtMartens} below), which gives an upper bound for the dimension of $B^k_{1, d}$. Various extensions of this result in the line bundle setting exist in the literature. For instance, Caporaso \cite{Ca} extended Martens' theorem to binary curves, and in \cite{Baj15} and \cite{Baj17} the second author obtained similar bounds for the dimensions of the secant schemes of very ample line bundles.

Moving to higher rank: The second author obtained in \cite[Theorem 1]{Baj20} a Martens-type bound for the dimension of $B^k_{2, d}$, and conjectured in \cite[Remark 1 (b)]{Baj20} an upper bound for $\dim B^k_{n, d}$ in terms of $g$, $n$, $d$ and $k$. A main object of the present work is to prove a refined version of this conjecture for $n \ge 3$.

A natural extension of a result of Feinberg \cite[Proposition 1]{Baj20} is a key tool in our argument, and leads us naturally for rank $n \ge 2$ to discuss Brill--Noether loci for two different types of bundle: \emph{first type}, where all sections of the bundle belong to a line subbundle, and \emph{second type}, where if $k \ge 2$ the generated subsheaf has rank at least $2$ (see Definition \ref{types} for a precise formulation). This dichotomy is visible also in \cite{BH}, where Brill--Noether loci were studied on moduli spaces of symplectic bundles via techniques similar to those of \cite{Baj20}.

Let us give a summary of the paper. In {\S} \ref{Bounds} we generalize the aforementioned result of Feinberg in Lemma \ref{lemma4}, and proceed to derive in Theorems \ref{boundthrmFirst} and \ref{boundthrmSecond} the Martens-type dimension bounds for $B^k_{n, d}$, for components where the generic element has first type and second type respectively. In the case of first type, as in \cite[{\S} 3]{BH} the bound depends on the degree of the generated line subbundle as well as on $g$, $n$, $d$ and $k$. In Example \ref{sharp}, we show that the bound is sharp in the case of first type, and in Corollary \ref{RefinedFirstType}, we obtain numerical restrictions imposed by the requirement that a component have general element of first type.

In {\S} \ref{gmo}, we give more refined results for $d \le g-1$. In particular, Theorem \ref{GenMumford} is a generalization of Mumford's theorem to rank $n \ge 2$. In the final section, we apply our dimension bounds in Theorem \ref{B2irr} to prove that if $n\geq 5$, then the locus $B^2_{n, n(g-1)}$ is irreducible and reduced for any curve $C$.

Avenues of future investigation include generalizing and refining the dimension bounds in Theorems \ref{boundthrmFirst} and \ref{boundthrmSecond}, as well as Theorems \ref{GenMumford} and \ref{B2irr}. We hope that these results and techniques may also be applicable to the study of \emph{generalized secant loci}, which are Brill--Noether-type loci on certain Quot schemes of bundles over $C$ studied in \cite{Hit}; in particular, investigating the validity of a Martens-type bound suggested in \cite[Conjecture 2.13]{Hit}.

\subsection*{Notation}
Throughout, $C$ will denote a complex projective smooth curve of genus $g \ge 3$. For any sheaf $F$ over $C$, we abbreviate $H^i (C, F)$ and $h^i (C, F)$ to $H^i (F)$ and $h^i (F)$ respectively.

\section{A Martens-type bound for higher rank bundles} \label{Bounds}

We recall firstly the notions of bundles of first and second type from \cite{Baj20}.

\begin{definition} \label{types}
Let $V$ be a bundle over $C$ with $h^0 (V) \ge 1$. Then $V$ is said to be \textsl{of first type} if $H^0 (V) = H^0 (L)$ for a line subbundle $L \subseteq V$. If $V$ has a line subbundle $M$ with $h^0 (M) = 1$, then $V$ is \textsl{of second type}.
\end{definition}

\begin{remark}
Let $V \to C$ be a vector bundle of rank $n \ge 2$ satisfying $h^0 (V) \ge 1$.
\begin{enumerate}
\renewcommand{\labelenumi}{(\roman{enumi})}
\item 
The bundle $V$ is of first type if and only if the generated subsheaf 
$$
\im \left( \Oc \otimes H^0 ( V ) \ \to \ V \right)
$$
has rank one.
\item It is easy to see that $h^0 (V) = 1$ if and only if $V$ is both of first type and of second type.
\end{enumerate}
\end{remark}

\begin{lem} \label{lemma4}
Suppose that $V$ is a bundle of rank $n \geq 2$ which is not of first type (in particular, this implies that $h^0 (V) \geq 2$). Then there exists a subbundle $H \subseteq V$ of rank $n-1$ with $h^0 (H) = 1$.
\end{lem}

\begin{proof}
By \cite[Lemma 1]{Baj20}, since $V$ is not of first type there exists an exact sequence $0 \rightarrow N \rightarrow V \rightarrow P \rightarrow 0$ where $N$ is a line bundle satisfying $h^0 (N) = 1$ and $P$ has rank $n-1$. If $n = 2$ then we may set $H = N$ and obtain the lemma.

If $n \geq 3$ then according to \cite[Proposition 6.1]{PR}, for $\beta \gg 0$ the bundle $P^*$ has a stable vector bundle quotient $G^*$ of rank $n - 2$ and degree $\beta$. The corresponding corank one subbundle $G \subset P$ is then stable and of negative degree, and therefore satisfies $h^0 (G) = 0$. Let $H$ be the preimage of $G$ in $V$, so that we have a diagram
\[ \xymatrix{
0 \ar[r] & N \ar[r] & V \ar[r] & P \ar[r] & 0 \\
0 \ar[r] & N \ar[r] \ar[u]^\wr & H \ar[r] \ar[u] & G \ar[r] \ar[u] & 0 .
} \]
As $h^0 (G) = 0$, we have $h^0 (H) = h^0 (N) = 1$, and we are done.
\end{proof}

As in \cite[{\S} 3]{BH}, we treat separately the cases where the general element of a component of $B^k_{n, d}$ is of first type and of second type.

\subsection{Components whose elements are of first type}

We recall now a familiar description of the tangent spaces to $B^k_{n, d}$; see for example \cite[{\S} 2]{GT}. Suppose that $E$ is a point of $B^k_{n, d} \setminus B^{k+1}_{n, d}$. Denote by
\[
\mu_E \colon H^0 (E) \otimes H^0 (K \otimes E^*) \ \rightarrow \ H^0 (K \otimes E \otimes E^*)
\]
the Petri map associated to $E$. Then
\[
T_E B^k_{n, d} \ = \ \left( \im \mu_E \right)^\perp \ \subseteq \ H^0(K\otimes E\otimes E^*)^* .
\]

\begin{thrm} \label{boundthrmFirst}
Let $X \subseteq B^k_{n, d}$ be a closed irreducible sublocus such that a general $V \in X$ satisfies $h^0 (V) = k$ and the subbundle generated by $H^0 (V )$ is a line bundle $M$ of degree $\ell$. For any such $V$, we have
$$
\dim X \ \leq \ \dim  T_V X \ \leq \ n(n-1)(g-1) + d - n \ell +\dim T_M B^k_{1, \ell} .
$$
\end{thrm}

\begin{proof}
We follow a strategy similar to that in \cite[Theorem 3.5]{BH}, but in fact simpler. Let $j \colon M \rightarrow V$ be the inclusion. Then there are maps
$$
j^* \colon H^1 (\End V) \ \rightarrow \ H^1 (\Hom(M, V)) \quad \hbox{and} \quad j_* \colon H^1 (\End M) \ \rightarrow \ H^1 (\Hom(M,V)) .
$$
Now a deformation $\bV$ of $V$ induces a deformation of $M$ if and only if there exists a commutative diagram
\[ \xymatrix{
0 \ar[r] & M \ar[r] \ar[d]_j & \bM \ar[r] \ar[d] & M \ar[r] \ar[d]_j & 0 \\
0 \ar[r] & V \ar[r] & \bV \ar[r] & V \ar[r] & 0
}\]
This is equivalent to the condition that
\begin{equation} \label{defequality}
 j_*\delta(\mathbb{M}) = j^*\delta (\mathbb{V}) \quad \in \quad H^1(\Hom(M,V)) .
\end{equation}
Now $M$ defines a point of $B^k_{1, \ell}$. The deformation $\bV$ corresponds to a tangent direction in $T_V X$ only if $V$ satisfies (\ref{defequality}) for some $\bM$ belonging to $T_M B^k_{1, \ell} \subseteq H^1 (\End M)$. It follows that
\begin{equation}\label{defequation2}
T_V X \ \subseteq \ \left( j^* \right)^{-1} \left( j_* \left( T_M B^k_{1,l} \right) \right) .
\end{equation}
Now clearly $j^*$ is surjective. As $h^0 (\Hom(M, V)) \geq 1$ (in fact one can show that there is equality), by (\ref{defequation2}) we obtain
$$
\dim T_V X \ \leq \ h^1 (\End V) - 1 + \chi( \Hom(M, V) ) + \dim T_M B^k_{1,\ell} .
$$
By Riemann--Roch, this becomes
$$
\dim T_V X \ \leq \ n(n-1)(g-1) + d - n \ell + \dim T_M B^k_{1, \ell} ,
$$
as desired.
\end{proof}


The following simple observation allows us to use the original Martens' theorem in cases of interest.

\begin{lem}[Martens' Theorem] \label{ExtMartens}
Assume that $C$  is a smooth curve of genus $g \geq 3$. Suppose that $2 \le \ell \le 2g - 4$ and $ 0 < 2k-2 \leq \ell $. Then $\dim B^k_{1, \ell} \le \ell - 2k + 2$, with equality if and only if $C$ is hyperelliptic.
\end{lem}

\begin{proof}
Note that $B^k_{1,\ell} = W_\ell^{k-1} (C)$ in the notation of \cite{ACGH}. If $\ell \le g-1$ then the statement is Martens' theorem \cite[p.\ 191 ff.]{ACGH} 
 If $g \le \ell \le 2g- 4$, then
\[
\dim B^k_{1, \ell} \ = \ \dim B^{k - \ell + g - 1}_{1, 2g - 2 - \ell} 
\]
by Serre duality. As $2 \le 2g - 2 - \ell \le g-1$, by Martens' theorem
\[
\dim B^{k - \ell + g - 1}_{1, 2g - 2 - \ell} \ \le \ (2g - 2 - \ell) - 2 (k - \ell + g - 1) + 2 \ = \ \ell - 2k + 2 ,
\]
with equality if and only if $C$ is hyperelliptic.
\end{proof}

\begin{cor} \label{cor1}
Suppose that $g \geq 2$ and $n \geq 2$ and $2 \leq d \leq 2n(g-1) - n$, and $k \ge 2$. If $X \subseteq B^k_{n, d}$ is a locus as in the hypothesis of Theorem \ref{boundthrmFirst}, then for $V \in X$, we have
$$
\dim_V X \ \leq \ \dim T_V X \ \leq \ n(n-1)(g-1) + d - 2k + 1 - 2(n-1) ,
$$
with equality only if $C$ is hyperelliptic. 
\end{cor}

\begin{proof}
Let $V$ be a bundle of first type in $B^k_{n,d}$, and $M \subset V$ the generated invertible subsheaf. As $V$ is stable, $\deg M =: \ell \le 2g - 4$. Thus by Martens' theorem and Lemma \ref{ExtMartens}, we have $\dim T_M B^k_{1, \ell} \leq \ell - 2k + 2$, with equality only if $C$ is hyperelliptic. By Theorem \ref{boundthrmFirst}, we obtain
\begin{equation} \label{inequalityoffirsttype}
\dim_V X \ \leq \ \dim T_V X \ \leq \ n(n-1)(g-1) + d - (n-1) \ell - 2k + 2 .
\end{equation}
As $k \geq 2$, we must have $\ell \geq 2$, with equality only if $C$ is hyperelliptic and $M$ is the hyperelliptic line bundle. The statement follows.
\end{proof}

Let us show by example that the bound in Theorem \ref{boundthrmFirst} is sharp.

\begin{example} \label{sharp}
Let $C$ be a hyperelliptic curve of genus $g \ge 3$ with hyperelliptic line bundle $H$, and suppose $n \ge 2$. Fix integers $k$ and $d$ satisfying $1 \le k < \frac{g}{2}$ and
\begin{equation} \label{SharpHyp}
0 \ < \ d - 2n (k - 1) \ \le \ n .
\end{equation}

Now let $L_0$ be a general element of $\Pic^{2(k-1)}(C)$, and $F_0$ a bundle of rank $n-1$ and degree $d - 2(k-1)$ which is general in moduli. By (\ref{SharpHyp}), we may apply \cite[Proposition 1.11]{RT}: If $0 \to L_0 \to E_0 \to F_0 \to 0$ is a general extension, then $E_0$ is a stable vector bundle. As stability is invariant under tensoring by line bundles, also the extension
\[ 0 \ \to \ H^{k-1} \ \to \ E_0 \otimes L_0^{-1} \otimes H^{k-1} \ \to \ F_0 \otimes L_0^{-1} \otimes H^{k-1} \ \to \ 0 \]
is a stable vector bundle. We set $E := E_0 \otimes L_0^{-1} \otimes H^{k-1}$ and $F := F_0 \otimes L_0^{-1} \otimes H^{k-1}$. Since $F_0$ is general in moduli, after deforming $F_0$ if necessary we may assume that
 $F$ is nonspecial; that is, $h^0 (F) \cdot h^1 (F) = 0$. Now by (\ref{SharpHyp}) and since $k < \frac{g}{2}$, we have
\[
\deg F \ = \ d - 2(k - 1) \ \le \ n + 2 (n-1) (k-1) \ < \ n + (n-1) (g-2) 
\ = \ (n-1)(g-1) + 1 .
\]
Thus $\mu (F) \le g-1$, and so $h^0 (F) = 0$. Therefore, $h^0 (E) = h^0 (H^{k-1}) = k$. From the vanishing of $H^0 (F)$ we also deduce that $h^0 ( \Hom (H^{k-1}, F)) = 0$. This means that by \cite[Lemma 3.3]{NR} the classifying map $\PP H^1 ( \Hom (F, H^{k-1})) \dashrightarrow \Und$ is injective; and moreover that a given general $E$ occurs in $\PP H^1 ( \Hom (F, H^{k-1}))$ for exactly one $F$.

By the above discussion and using Riemann--Roch, we compute that locus of such extensions in $\Und$ is of dimension
\begin{equation} \label{SharpBound}
h^1 (\End F) + h^1 ( \Hom (F , H^{k-1} )) - 1 \ = \ 
n (n-1) (g-1) + \left( d - n \cdot \deg H^{k-1} \right) .
\end{equation}
Now using \cite[p.\ 13]{ACGH}, one can check that $B^k_{1, 2k-2}$ is the point $H^{k-1}$ and that the Petri map
\[
\mu_{H^{k-1}} \colon H^0 ( H^{k-1} ) \otimes H^0 ( K \otimes H^{1-k} ) \ \to \ H^0 (K)
\]
is surjective. Thus $\dim T_{H^{k-1}} B^k_{1, 2k-2} = 0$. Thus (\ref{SharpBound}) is exactly the bound in Theorem \ref{boundthrmFirst}.
\end{example}

Before proceeding, we recall a definition.

\begin{definition} \label{exceptional}
Let $C$ be a nonhyperelliptic curve of genus $g \geq 4$. Then $C$ will be called an \textsl{exceptional} curve if it is trigonal, bielliptic, or a smooth plane quintic.
\end{definition}

\begin{cor} \label{RefinedFirstType}
Let $C$ be a curve of genus $g \ge 4$ which is neither hyperelliptic nor exceptional and assume that $3 \leq k \leq n$. Suppose that $2 \le d \le 2n(g-1) - n$.
\begin{enumerate}
\renewcommand{\labelenumi}{(\roman{enumi})}
\item Let $E$ be a general element of any component of $B^k_{n,n(g-1)}$. Then for any line subbundle $M \subset E$ we have $h^0 (M) < \frac{k}{2}$.
\item Set $a := n(g-1)-d$. Then there is no component $X$ of $B^k_{n,d}$ with 
\begin{equation} \label{strangeinequality}
2n - k^2 \ > \ (k-1) (a-2)
\end{equation}
having general element of first type. In particular, if
$$
d \ \in \ \{ n(g-1), n(g-1) - 1 \}
$$
and $E$ is a general element of any component of $B^2_{n, d}$, then $h^0 (N) \le 1$ for any line subbundle $N \subset E$.
\end{enumerate}
\end{cor}

\begin{proof}
(i)	Suppose that $X$ is an irreducible component of $B^k_{n, n(g-1)}$. If any line subbundle $L$ of $E$ satisfies $h^0 (L) = 1$, then the assertion holds because $k \geq 3$. If $E$ admits a line subbundle $M$ of degree $\ell$ with $h^0 (M) = k_0 \ge 2$, then a similar argument to that in the proof of Theorem \ref{boundthrmFirst} shows that 
\begin{align*}
\dim X \ &\leq \ n(n-1)(g-1) + d - n \ell + \dim T_M B^{k_0}_{1, \ell} \\
&\le \ \dim \Und - \left( n(g-1) - d + n \ell - \dim T_M B^{k_0}_{1, \ell} + 1 \right) .
\end{align*}
On the other hand, since $X$ has expected codimension $k (n (g-1) - d + k)$ in $\Und$, we have
\begin{align}\label{1}
n(g-1) - d + n \ell + 1 \ \leq \ k(n(g-1) - d + k) + \dim T_M B^{k_0}_{1, \ell} .
\end{align}
\noindent As $C$ is not hyperelliptic or exceptional and since $\ell \leq g-2$ by stability of $E$, we can apply Mumford's theorem \cite[p.\ 193 ff.]{ACGH} to $T_M B^{k_0}_{1, \ell}$. So we obtain
\[
\dim T_M B^{k_0}_{1, \ell} \ \leq \ \ell - 2(k_0 - 1) - 2 .
\]
As $d = n(g-1)$, the above together with (\ref{1}) implies that
\[
n \ell + 1 \ \le \ k^2 + \ell - 2 k_0 .
\]
As by hypothesis $B^{k_0}_{1, \ell}$ is nonempty, again using Mumford's theorem we obtain $\ell \ge 2 k_0$. Therefore, the above inequality yields
 $2nk_0 + 1 \ \le \ k^2$. As we have assumed that $n \ge k$, we obtain $k_0 < \frac{k}{2}$ as desired.

(ii) Let $X$ be a component of $B^k_{n, d}$ whose general element is of first type. By Corollary \ref{cor1}, we have
$$
\dim X \ \leq \ n^2 (g-1) + 1 - n(g-1) + d - 2k - 2(n-1) .
$$
Comparing with the expected codimension as in (i), we obtain
$$
n(g-1) - d + 2k + 2(n-1) \ \le \ k ( k - d + n(g-1) ) .
$$
Substituting $a = n(g-1) - d$, this becomes
$$
0 \ \leq \ k ( a + k ) - a - 2k -2(n-1) ;
$$
that is, $2n - k^2 \le (a-2)(k-1)$. Thus if (\ref{strangeinequality}) holds, then no such component $X$ exists. For the rest: As $g \ge 4$, we check easily that $n(g-1) \le 2n(g-1) - n$. The statement follows.
\end{proof}

\subsection{Components containing bundles of second type}

\begin{thrm} \label{boundthrmSecond}
Let $C$ be a curve of genus $g \ge 3$. Let $k$ and $d$ be integers satisfying $3 \leq d \leq 2n(g-1) - 3$ and $2 \leq k \leq \frac{d}{2} + n$. Let $X$ be a component of $B^k_{n, d}$ whose general element is of second type. Then
$$
\dim X \ \leq \ n(n-1)(g-1) + d - 2k + 1 .
$$
\end{thrm}

\begin{proof}
Assume firstly that $d \le n(g-1)$. Suppose for a contradiction that there is a component $X$ whose general element is of second type and which satisfies
\[
\dim X \ \geq \ n(n-1)(g-1) + d - 2k + 2.
\]
Now by \cite[Proposition 1.6]{Lau} and \cite[Remark 2.3]{CT}, no irreducible component of $B^k_{n,d}$ is contained entirely in $B^{k+1}_{n,d}$. Thus we may choose $V \in X$ satisfying $h^0 (V) = k$ and which is of second type. By Lemma \ref{lemma4} we can represent $V$ as an extension $0 \to H \to V \to L \to 0$ where $h^0 (H) = 1$ and $L$ is a line bundle. Then we obtain a diagram
\begin{equation} \label{maindiagram}
\text{\footnotesize
\xymatrix{
H^0 (H) \otimes H^0 (K \otimes V^{\ast}) \ar[r]^{f_1} \ar[d]^\mu & H^0 (V) \otimes H^0 (K \otimes V^{\ast}) \ar[r]^{g_1}\ar[d]^{\mu^n_V} & W \otimes H^0 (K \otimes V^{\ast}) \ar[d]^{\mu_{L, V}} \\
H^0 (K \otimes H \otimes V^{\ast}) \ar[r]^{f_2} & H^0 (K \otimes V \otimes V^{\ast}) \ar[r]^{g_2} & H^0 (K \otimes L\otimes V^{\ast})
}} \end{equation}
where $W$ is the image of the map $H^0 (V) \rightarrow H^0 (L)$ and in which the maps $f_1$, $f_2$ and $\mu$ are injective and $g_1$ is surjective. The snake lemma applied to this situation implies that $\dim \ker \mu^n_V \leq \dim  \ker \mu_{L, W}$. By our assumption on $\dim X$, we have
\begin{equation} \label{inequality2}
\dim \ker \mu_{L, W} \ \geq \ (k-1) \cdot \left( n(g-1) - d + k - 1 \right) .
\end{equation}
Let $w_1, \ldots , w_{k-1}$ be a basis for $W$. For $2 \le i \le k-1$, set $W_i := \langle w_1 , \ldots , w_i \rangle$. Then for $3 \le i \le k-1$, we have
\[
\dim \ker \mu_{L, W_i} - \dim \ker \mu_{L, W_{i-1}} \ \leq \ h^0 ( K\otimes V^{\ast} ) .
\]
These inequalities together with the base point free pencil trick applied to the map
\[
\mu_{L, W_2} \colon W_2 \otimes H^0 (K \otimes V^{\ast}) \ \longrightarrow \ H^0 (K \otimes L \otimes V^{\ast})
\]
imply that
\begin{equation} \label{ineauality3}
h^0 (K \otimes V^{\ast} \otimes L^{\ast}(B)) \ \geq \ 2 \left( n(g-1) + k - d \right) - k + 1 ,
\end{equation}
where $B$ is the base locus of the pencil $\PP W_2 \subseteq |L|$.

Now since $h^0 (L) \geq k-1 \geq 1$, the line bundle $L$ can be written as $L = \Oc (D)$ for some effective divisor $D$ on $C$. Furthermore, since $B$ is the base locus of a pencil in $|L|$, we have $D=B + D_1$ for an effective $D_1$, and $L^{\ast}(B)=\mathcal{O}(-D_1)$. Therefore,
\[
H^0 (K \otimes V^{\ast} \otimes L^{\ast}(B)) \ \cong \ H^0 (K \otimes V^{\ast}(-D_1))
\]
can be identified with a subspace of $H^0 (K \otimes V^{\ast})$. It follows that $h^0 (K \otimes V^{\ast}) \geq h^0 (K \otimes V^{\ast}\otimes L^{\ast}(B))$. This, by (\ref{ineauality3}) together with the Riemann--Roch theorem applied to $K\otimes V^{\ast}$, gives $n(g-1) + k - d \geq 2 \left( n(g-1) + k - d \right) - k + 1$.
So 
\begin{equation} \label{inequality5}
n(g-1) - d \ \leq \ -1 ,
\end{equation}
contrary to hypothesis. Therefore, $\dim X \le n(n-1)(g-1) + d - 2k + 1$.

On the other hand, suppose that $d \ge n(g-1)$. By Serre duality (as in the proof of Lemma \ref{ExtMartens}), the map $V \mapsto K \otimes V^*$ gives an identification of $X$ with a component $\hX$ of $B^{k-d+n(g-1)}_{n, 2n(g-1)-d}$. If every element of $\hX$ is of first type, then by Theorem \ref{boundthrmFirst} we have
$$
\dim X \ = \ \dim \hX \ \le \ n(n-1)(g-1) + d - (n-1) \ell - 2k + 2 \ \le \ n(n-1)(g-1) + d - 2k + 1 ,
$$
the latter inequality since $n \ge 2$ and $k \ge 2$ and using stability. On the other hand, if a general element of $\hX$ is of second type, then we can apply the previous argument to $\widehat{X}$, and obtain
$$
\dim X \ = \ \dim \widehat{X} \ \le \ n(n-1)(g-1) + \left( 2n(g-1) - d \right) - 2 \cdot \left( k - d + n(g-1) \right) + 1 ,
$$
which simplifies to $\dim X \le n(n-1)(g-1) + d - 2k + 1$ as desired.
\end{proof}

\begin{remark}
The proof of the bound on $\dim T_E B^k_{n, d}$ obtained above for $2 \le d \le n(g-1)$ does not use the stability of $V$.
\end{remark}

\begin{remark} \label{kEqualsOne}
Theorems \ref{boundthrmFirst} and \ref{boundthrmSecond} do not apply when $k = 1$. In this case the generated subsheaf is isomorphic to $\Oc$, and the Petri map can be identified with
$$
H^0 ( \Oc ) \otimes H^0 ( K \otimes V^* ) \ \to \ H^0 ( K \otimes V^* ) ,
$$
which is clearly injective. It follows that $B^1_{n, d}$ is smooth and of expected dimension for all $n$ and $d$.
\end{remark}

In the next section, we shall investigate some special cases in more depth.

\section{The case \texorpdfstring{$d\leq g-1$}{d <= g-1}} \label{gmo}

The following is an adaptation of \cite[Lemma 4.1]{BGN}, both in proof and in its statement.

\begin{lem} \label{lemma5}
Let $C$ be a curve which is neither hyperelliptic nor exceptional (cf.\ Definition \ref{exceptional}). Then any bounded set $\cF$ of non-stable bundles of rank $n$ and degree $\gamma$ with at least $k$ independent global sections depends on at most
$$ \begin{cases}
n(n-1)(g-1) + \gamma \hbox{ parameters if } k = 1 ; \\
n(n-1)(g-1) + \gamma - 2k \hbox{ parameters if } k \geq 2 .
\end{cases} $$
\end{lem}

\begin{proof}
Our argument is a straightforward adaptation of the proof of \cite[Lemma 4.1]{BGN}. Suppose firstly that $n = 2$, and let $\beta_2$ be the number of parameters on which elements of $\cF$ depend. Every $P \in \cF$ is an extension $0 \to L_1 \to E \to L_2 \to 0$ for some $L_i \in \Pic^{d_i} (C)$ where $d_1 \geq d_2$ and $h^0 (L_1) + h^0 (L_2) \geq k$. If $h^0 (L_1) = k$ then $L_2$ varies in an open subset of $\Pic^{d_2} (C)$. If $k = 1$, then 
$$
\beta_2 \ \leq \ d_2 + g + \left( h^1 ( L_1 L_2^{-1}) - 1 \right) .
$$
If $k \ge 2$ then by Mumford's theorem \cite[p.\ 193 ff.]{ACGH} we have
$$
g + \left( d_1 - 2k \right) + \left( h^1 (L_1 L_2^{-1}) - 1 \right) .
$$
As $h^0 ( L_1 ) > h^0 ( L_2 )$ for a general pair $(L_1 , L_2)$ as above, we have $L_1 \neq L_2$; whence $h^0 ( K \otimes L_2 \otimes L_1^{-1}) \leq g-1$ and we obtain the assertion in this case. The cases in which $L_1$ is a general line bundle or both $L_1$ and $L_2$ have nonzero sections are similar.

For $n \geq 3$ we use induction on $n$ together with an argument similar to that in \cite[Lemma 4.1]{BGN}. We omit the details. 
\end{proof}


\begin{thrm} \label{GenMumford}
Let $C$ be a curve of genus $g \geq 2$. If for some $d \leq g-1$ one has $\dim B^k_{n, d} \ge n(n-1)(g-1) + d - 2k + 1$, then either $C$ is hyperelliptic or an exceptional curve.
\end{thrm}
\begin{proof}
Assume that $C$ is neither hyperelliptic nor exceptional. We shall prove by induction on $n$ that $\dim B^k_{n, d} < n(n-1)(g-1) + d - 2k + 1$. For $n = 1$, the statement is exactly Mumford's theorem \cite[p.\ 193 ff.]{ACGH}

Suppose that $n \ge 2$, and let $X$ be an irreducible component of $B^k_{n, d}$. If $X$ consists entirely of bundles of first type, then the inequality (\ref{inequalityoffirsttype}) implies the assertion (indeed, without using the induction hypothesis). So we may assume that a general element of $X$ is of second type. Thus a general $E \in X$ is an extension $0 \to M \to E \to P \to 0$ where $M$ is a line bundle of degree $d_M \ge 0$ satisfying $h^0 (M) = 1$. As $E$ is stable, $d_M \le g-2$, and the locus of effective $M \in \Pic^{d_M} (C)$ is of dimension $d_M$. Furthermore, $h^0 (P) \ge k - 1$. If $P$ is not stable then by Lemma \ref{lemma5} it depends on at most $(n-1)(n-2)(g-1) + (d - d_M) - 2 (k-1)$ parameters. If $P$ is stable, then $P \in B^{k-1}_{n-1, d - d_M}$, and so by induction the same is true.

Now in an exact sequence $0 \to M' \to E' \to P' \to 0$, if $\alpha \colon P' \to M'$ is a nonzero map, then the composition $E' \to P' \xrightarrow{\alpha} M' \to E'$ is a nonzero map which is not of the form $\lambda \cdot \Iden_{E'}$. Thus $E'$ is not simple, and in particular it is not stable. Thus for any $0 \to M \to E \to P \to 0$ as in the previous paragraph, we have 
$h^0 ( \Hom (P, M)) = 0$. Therefore, by Riemann--Roch,
$$
h^1 (\Hom (P, M)) \ = \ 
d - n d_M + (n-1) (g-1) .
$$
Furthermore, the cup product map $H^1 (\Hom(P, M)) \to \Hom (H^0 (P), H^1 (M))$ is clearly nonzero. Thus a general extension $0 \to M \to E_1 \to P \to 0$ satisfies $h^0 (E_1) \le k - 1$.

It follows that $E \in X$ depends on at most
\begin{multline*}
d_M + \left( (n-1)(n-2)(g-1) + (d - d_M) - 2 (k-1) \right) + \\
\left( d - n d_M + (n-1) (g-1) - 2 \right) \ = \\
\left( n(n-1)(g-1) + d - 2k \right) - \left( (n-1) (g-1) + n d_M - d \right)
\end{multline*}
parameters. As $d_M \ge 0$ and $d \le g-1$ and $n \ge 2$, we have $(n-1) (g-1) + n d_M - d \ge 0$ and we obtain the desired dimension bound on $X$.
\end{proof}

\section{Irreducibility of \texorpdfstring{$B^2_{n, n(g-1)}$}{B(2 ; n, d}}

In this final section, we give an application of our dimension bounds. According to Theorems \ref{boundthrmFirst} and \ref{boundthrmSecond}, the locus $B^2_{n, n(g-1)}$ attains its expected dimension $n^2(g-1) - 3$. We now show that $B^2_{n, n(g-1)}$ is in fact integral.

\begin{thrm} \label{B2irr}
For any curve $C$, if $n \ge 5$ then the locus $B^2_{n, n(g-1)}$ is irreducible and reduced.
\end{thrm}

\begin{proof}
Since $B^2_{n, n(g-1)}$ has dimension $n^2 (g-1) - 3$, in order to prove both irreducibility and reducedness it would suffice to show that
$$
\dim \Sing \left( B^2_{n, n(g-1)} \right) \ \leq \ n^2 (g-1) - 5 .
$$
We assume for a contradiction that there exists a component $\Xi$ of $\Sing \left( B^2_{n, n(g-1)} \right)$ such that $\dim \Xi \ge n^2 (g-1) - 4$.

Firstly, by Theorems \ref{boundthrmFirst} and \ref{boundthrmSecond}, all components of $B^3_{n , n(g-1)}$ have dimension at most
$$
n(n-1)(g-1) + n(g-1) - 2 \cdot 3 + 1 \ = \ n^2 ( g-1 ) - 5 .
$$
Therefore, we may assume for a general $V \in \Xi$ that $h^0 (V) = 2$ and that the Petri map $\mu \colon H^0 (V) \otimes H^0 (K \otimes V^*) \to H^0 (K \otimes \End V)$ has nonzero kernel.

Now if $V \in B^2_{n, n(g-1)}$ is of second type, then both the morphisms $\mu$ and $\mu_{L, V}$ in diagram (\ref{maindiagram}) are injective, and $\mu^n_V$ is injective by the Five Lemma. Thus a general $V \in \Xi$ must be of first type. In particular, a general $V \in \Xi$ has a line subbundle $L_V$ with $h^0 (L_V) = h^0 (V) = 2$.

Let $x_0$ be a fixed general point of $C$. Then since $L_V$ has rank one,
$$
h^0 (V \otimes \Oc (-x_0)) \ = \ h^0 (L_V (-x_0)) \ \ge \ 1
$$
for general $V \in \Xi$. Thus the association $V \mapsto V(-x_0)$ defines a morphism $\Xi \to B^1_{n, n(g-1)-n}$ which is clearly injective. It follows that $B^1_{n, n(g-1)-n}$ has a component of dimension at least $\dim \Xi = n^2 (g-1) - 4$. But by Remark \ref{kEqualsOne}, every component of $B^1_{n, n(g-1)-n}$ has dimension $n^2 (g-1) - n$. Therefore, $n \leq 4$ and this gives a contradiction to our assumption on $n$. Thus no such $\Xi$ can exist. The theorem follows.
\end{proof}

\noindent \textbf{Acknowledgment:} Ali Bajravani is supported by a George Forster Fellowship at Humboldt-Universität zu Berlin. He would like to express his gratitude to Gavril Farkas and to the Alexander von Humboldt Foundation for their support.


\end{document}